\documentclass[a4paper]{amsart}

\usepackage[utf8]{inputenc}
\usepackage{amsmath}
\usepackage{amsfonts}
\usepackage{xypic}
\usepackage{amsthm}
\usepackage{mathrsfs}
\usepackage{amssymb}
\usepackage[all]{xy}
\xyoption{rotate}
\usepackage{color}
\usepackage{hyperref}
\usepackage{graphicx}
\usepackage{wrapfig}
\usepackage{float}

\graphicspath{ {dibujos/} }

\numberwithin{equation}{section}

\newtheorem{theorem}{Theorem}[section]
\newtheorem{lemma}[theorem]{Lemma}
\newtheorem{proposition}[theorem]{Proposition}
\newtheorem*{proposition*}{Proposition}
\newtheorem*{theorem*}{Theorem}
\newtheorem*{corollary*}{Corollary}
\newtheorem*{lemma*}{Lemma}

\newtheorem{corollary}[theorem]{Corollary}
\newtheorem{definition}[theorem]{Definition}

\theoremstyle{remark}
\newtheorem{remark}[theorem]{Remark}

\newtheorem{innercustomgeneric}{\customgenericname}
\providecommand{\customgenericname}{}
\newcommand{\newcustomtheorem}[2]{%
  \newenvironment{#1}[1]
  {%
   \renewcommand\customgenericname{#2}%
   \renewcommand\theinnercustomgeneric{##1}%
   \innercustomgeneric
  }
  {\endinnercustomgeneric}
}

\newcustomtheorem{customthm}{Theorem}
\newcustomtheorem{customprop}{Proposition}
\newcustomtheorem{customlemma}{Lemma}

\newtheoremstyle{TheoremNum}
        {\topsep}{\topsep}              
        {\itshape}                      
        {}                              
        {\bfseries}                     
        {.}                             
        { }                             
        {\thmname{#1}\thmnote{ \bfseries #3}}
    \theoremstyle{TheoremNum}

\renewcommand{\Im}{\mathrm{Im}}

\DeclareMathOperator{\Ker}{Ker}

\newcommand{\Spec}{\mathrm{Spec}}

\newcommand{\End}{\mathrm{End}}
\newcommand{\PEnd}{\mathrm{ParEnd}}
\newcommand{\SPEnd}{\mathrm{SParEnd}}

\newcommand{\ad}{\mathrm{ad}}
\newcommand{\gr}{\mathrm{gr}}

\newcommand{\rk}{\mathrm{rk}}

\newcommand{\GL}{\mathrm{GL}}

\newcommand{\parmu}{\mathrm{par-}\mu}
\newcommand{\pardeg}{\mathrm{par-}\deg}

\newcommand{\Ee}{\mathcal{E}}
\newcommand{\Ff}{\mathcal{F}}

\newcommand{\Mm}{\mathcal{M}}
\newcommand{\Nn}{\mathcal{N}}
\newcommand{\Oo}{\mathcal{O}}
\newcommand{\Pp}{\mathcal{P}}
\newcommand{\Qq}{\mathcal{Q}}

\newcommand{\hM}{\widehat{\M(\Pp,\alpha)}}
\newcommand{\hMn}{\widehat{\M(\Pp,\alpha)}_{nilp}}
\newcommand{\hMs}{\widehat{\M(\Pp,\alpha)}_{st}}
\newcommand{\Mn}{{\M}(\Pp,\alpha)_{nilp}}
\newcommand{\Ms}{{\M}(\Pp,\alpha)_{st}}
\newcommand{\MO}{{\M}(\Pp,\alpha)_{\Oo}}
\newcommand{\VE}{\V_{\Ee}}
\newcommand{\Vn}{\V_{\Ee,nilp}}
\newcommand{\VO}{\V_{\Ee,\Oo}}
\newcommand{\Vst}{\V_{\Ee,st}}

\newcommand{\oVE}{\ol{\V}_{\Ee}}
\newcommand{\oVn}{\ol{\V}_{\Ee,nilp}}

\newcommand{\Un}{\mathbf{Un}}
\newcommand{\Unn}{\mathbf{Un}_{nilp}}
\newcommand{\Uns}{\mathbf{Un}_{st}}
\newcommand{\Sh}{\mathbf{S}}
\newcommand{\Shn}{\mathbf{S}_{nilp}}
\newcommand{\Shs}{\mathbf{S}_{st}}

\newcommand{\hrn}{\hat{r}_n}
\newcommand{\hrs}{\hat{r}_s}
\newcommand{\Ex}{\mathbf{Ex}}
\newcommand{\Wo}{\mathbf{W}}

\newcommand{\C}{\mathrm{C}}

\renewcommand{\H}{\mathrm{H}}

\newcommand{\M}{\mathrm{M}}
\newcommand{\N}{\mathrm{N}}

\newcommand{\V}{\mathrm{V}}

\newcommand{\ol}[1]{\overline{#1}}

\renewcommand{\AA}{\mathbb{A}}

\newcommand{\CC}{\mathbb{C}}

\newcommand{\HH}{\mathbb{H}}

\newcommand{\PP}{\mathbb{P}}

\newcommand{\projects}{\twoheadrightarrow}
\newcommand{\plonge}{\hookrightarrow}

\newcommand\Quotient[2]{
        \mathchoice
            {
                \text{\raise1ex\hbox{\thinspace $#1$}\Big{/} \lower1ex\hbox{$#2$} \thinspace}%
            }
            {
                #1\,/\,#2
            }
            {
                #1\,/\,#2
            }
            {
                #1\,/\,#2
            }
    }

\newcommand\GIT[2]{
        \mathchoice
            {
                \text{\raise1ex\hbox{\thinspace $#1$}\Big{/}\!\!\!\!\Big{/} \lower1ex\hbox{$#2$} \thinspace}%
            }
            {
                #1\,/\,#2
            }
            {
                #1\,/\,#2
            }
            {
                #1\,/\,#2
     a       }
    }

\title[Equality of the wobbly and shaky loci]{Equality of the wobbly and shaky loci}

\author{Ana Pe\'on-Nieto}
\address{Ana Pe\'on-Nieto	\newline\indent
	Department of Mathematics \newline\indent
 Universidade de Santiago de Compostela \newline\indent
 Spain \newline\indent
 and\newline\indent
 School of Mathematics \newline\indent University of Birmingham\newline\indent Watson Building, Edgebaston\newline\indent Birmingham B15 2TT, UK}
\email{ana.peon@usc.es, 
	a.peon-nieto@bham.ac.uk}

\thanks{The author was supported by the Beatriu de Pin\'os grant n. 2018 BP 332, under the scheme H2020-MSCA-COFUND-2017 Agreement n. 801370, the project GoH funded by the scheme H2020-MSCA-IF-2019,  Agreement n.	897722, and the AEI Consolidaci\'on Investigadora grant no. CNS2022-136042.}


\begin{document}
 
\maketitle
\begin{abstract}
Let $X$ be a smooth complex projective curve of genus $g\geq 2$, and let $D\subset X$ be a reduced divisor. 
We prove that a parabolic vector bundle
$\Ee$ on $X$ is (strongly) wobbly, i.e.  $\Ee$ has a non-zero (strongly) parabolic nilpotent Higgs field, if and only if it is (strongly)  shaky, i.e., it is in the image of the exceptional divisor of a suitable resolution of the rational map from the (strongly) parabolic Higgs moduli to the parabolic bundle moduli space, both assumed to be smooth. This solves a conjecture by Donagi–Pantev \cite{DPLectures} in the parabolic and the vector bundle context. To this end, we prove the stability of strongly very stable parabolic bundles, and criteria for very stability of parabolic bundles.
\end{abstract}

\section{Introduction}
Let $X$ be a smooth complex projective curve of genus $g\geq 2$ and let $K$ be its canonical bundle. Consider the moduli spaces $\N(n,d)$ of vector bundles of rank $n$ and degree $d$, and  $\M(n,d)$ of Higgs bundles with the same invariants. These are schemes parametrising vector bundles (resp. Higgs bundles, namely pairs $(E,\varphi)$ where $E$ is a vector bundle and $\varphi\in H^0(X,\End(E)\otimes K)$ is the Higgs field) of fixed rank and degree. This article is concerned with two loci in $\N(n,d)$: the wobbly and the shaky loci. Wobbly bundles are semistable bundles in $\N(n,d)$ admitting a non-zero nilpotent Higgs field. Shaky bundles are more involved to define. To do this, consider the dense open embedding $T^*\N^s(n,d)\subset\M(n,d)$, where $\N^s(n,d)$ denotes the stable locus in $\N(n,d)$. This yields a surjective rational map $\xymatrix{r:\M(n,d)\ar@{-->}[r]&\N(n,d)}$, which is a morphism away from the locus $\Un\subset\M(n,d)$ of semistable Higgs bundles whose underlying bundle is unstable.  Suppose that the moduli space of bundles is smooth. Then, by Hironaka's theorem on the elimination of indeterminacies \cite{H}, a finite number of blowups allows to resolve $r$ to a morphism $\xymatrix{\hat{r}:\hat{\M}(n,d)\ar[r]&\N(n,d)}$. An element in the image  under $\hat{r}$ of the exceptional divisor  $\Ex$ is called a shaky bundle \cite{DPLectures}.

The name wobbly first appeared in \cite{DPLectures}, but the interest in these objects dates back to the works of Drinfeld and Laumon from the 80's \cite{Drinfeld,L}. The focus of the earlier works was on the non wobbly bundles, also known as very stable bundles. Very stable bundles are open dense in the moduli space of bundles (as in particular they are always stable). Hence, they play a crucial role in the geometry of the global nilpotent cone and, by extension,  that of the moduli space of Higgs bundles. Investigation has however more recently also concentrated on wobbly bundles, after the conjecture by Donagi--Pantev that these are precisely the shaky bundles \cite[\S 6.1]{DPLectures}.  
Shaky bundles appear naturally in their programme to prove geometric Langlands via abelianization of Higgs bundles and non abelian Hodge theory. In the words of the authors \cite[\S 6.1]{DPLectures} ``A major step towards carrying out our program is the identification of shaky bundles''. Indeed, their programme aims at constructing the Hecke eigensheaves predicted by geometric Langlands via simpler abelianized ones \cite{DPLanglands}. The nature of the construction itself implies that these eigensheaves have to be singular along the shaky divisor. This is the case of the sphere minus five points \cite{DPHecke}, where the Hecke eigensheaves are proven to have polar singularities along the shaky locus, as expected after Drinfeld and Laumon's work \cite{LaumonGLn}. The equality of the shaky and  wobbly loci \cite[Prop. 5.10]{DPHecke}, done by an explicit description of the objects involved, proves crucial in the identification of the singularities of the sheaves \cite[\S 6.3.2]{DPHecke}.   

For general smooth Riemann surfaces, the geometric criterion for very stability proven by  Pauly and the author  \cite{PP} suggests the equality of the wobbly and shaky locus. This criterion was subsequently generalised to principal bundles by H. Zelaci \cite{Hacen}. 

The above applies to usual Higgs bundles. However, since the beginning of their programme, it was understood by Donagi--Pantev that parabolic bundles were the right setup to prove geometric Langlands via Higgs bundles \cite[\S 1.3]{DPHecke}, \cite{DPS}. This is why we have chosen to concentrate on the parabolic setup in the present proof of the equality of wobbly and shaky bundles. In fact, the arguments in this paper can be adapted to prove the equality of wobbly and shaky bundles with no parabolic structure, but this result follows also as a corollary of the parabolic case hereby presented. Let us point out that unlike in the case of $\PP^1$ minus five points, there is no elementary explicit description of the objects involved, and thus the techniques hereby used are totally different in nature. As an intermediate step, we generalise the results in \cite{PP, Hacen} to this setup.

\subsection*{Structure of the paper and summary of results}
Let $D\subset X$ be a reduced divisor,  let $\mathcal{P}:=\{P_x:x\in D\}$ be a set of conjugacy classes of parabolic subgroups of $\GL(n,\C)$ indexed by $D$, identified with partitions $\{\ol{m}_x=(m_{x,1},\dots, m_{x,r_x}),x\in D\}$ of the rank $n$. A quasi-parabolic vector bundle of rank $n$ and flag type $\Pp$ is a pair $\Ee=(E,\sigma)$ where $E$ is a rank $n$ vector bundle and $\sigma$ is a reduction of the structure group to the given parabolic subgroup at each point of $D$. Equivalently, we may identify $\sigma$ with a set of flags $\{E_{x,0}=0\subseteq E_{x,1}\subsetneq\dots\subsetneq E_{x,r_x}=E_x\}_{x\in D}$. Let $\alpha_x=(\alpha_{x,1},\dots,\alpha_{x,r_x})$ be an increasing  $r_x$-uple of non-positive real numbers in the interval $(-1,0]$. The assignment of the weights $\alpha_{x,i}$ to the steps $E_{x,i}$ of the flag defines a parabolic structure on the quasi-parabolic bundle $\Ee$. A parabolic Higgs bundle is a pair $(\Ee,\varphi)$ where $\Ee$ is a parabolic vector bundle and $\varphi\in H^0(X,\PEnd(\Ee)\otimes K(D))$, with $\PEnd(\Ee)\subset\End(E)$ the subsheaf of endomorphism of $E$ preserving the flags at the prescribed points.

 Each $\alpha$ as above defines a stability notion, yielding  moduli spaces of parabolic bundles $\N(\Pp,\alpha)$ and parabolic Higgs bundles $\M(\Pp,\alpha)$. 
The latter is Poisson with symplectic leaves given by Higgs fields whose residues belong to fixed adjoint orbits in the associated Levi  group. Particularly important are the strongly parabolic leaf (corresponding to the zero orbit), and the regular semisimple leaves (which are generic). Strongly parabolic Higgs bundles have Higgs fields belonging to the  subsheaf $\SPEnd(\Ee)\subset\PEnd(\Ee)$ of strongly parabolic endomorphisms, i.e., endomorphism strictly compatible with the filtration.

Consider the Hitchin fibration
$$
     h_{\Pp,\alpha} : \M(\Pp,\alpha) \longrightarrow  \H_D = \bigoplus_{i=1}^n \H^0(X, K^i(iD)), 
$$ 
defined by associating to a parabolic Higgs bundle $(\Ee, \varphi)$ the coefficients of the characteristic polynomial of $\varphi$.  If $\Ee$ is a stable parabolic vector bundle of type $(\Pp,\alpha)$, there is an embedding \[\H^0(X,\PEnd(\Ee)\otimes K(D))\hookrightarrow\M(\Pp,\alpha)\] given by $\Ee\mapsto(\Ee,0)$.  Motivated by \cite{PP,Hacen}, we consider the restriction $h_{\Ee,st}$ (respectively $h_{\Ee,nilp}$) of the Hitchin map to the vector space  $V_{\Ee, st}=\H^0(X,\SPEnd(\Ee)\otimes K(D))$ of strongly parabolic Higgs fields (resp.   $V_{\Ee, nilp}=\H^0(X,\PEnd(\Ee)\otimes K(D))_{nilp}$, of Higgs fields with nilpotent residue with respect to the parabolic subgroup given by the flag). If $\Ee$ is stable, this corresponds to the intersection of $\H^0(X,\PEnd(\Ee)\otimes K(D))$ with the strongly parabolic symplectic leaf (resp., the closure of the regular nilpotent symplectic leaf). These and other notions are explained in Section \ref{sec:prel}. 

In Section \ref{sec:criteria} we define the notion of a strongly very stable parabolic  bundle (respectively, very stable), that is, a parabolic bundle admitting no non-zero strongly parabolic Higgs field (resp. no non-zero  nilpotent parabolic Higgs field). Paralleling the case of vector bundles \cite{L},  we prove:
\begin{lemma*}[Lemma \ref{lm:vs_is_stable}]
  Let $\Ee$ be parabolic of type $(\Pp,\alpha)$. If it is strongly very stable, then it is $\alpha$-stable. In particular, if it is very stable, then it is $\alpha$-stable.    
\end{lemma*}
We note that this fact is implicit in the assertion in \cite[\S5.4]{DPHecke} that all results from \cite{L} adapt to the strongly parabolic setup in virtue of \cite{BiswasOrbi,Borne}. We hereby include a detailed proof of the aforementioned lemma, improving a result by \cite{SWW} (see Remark \ref{rk:SWW}). 

Similarly, the following result extends \cite[Theorem 1]{PP} to the parabolic setup.
 \begin{theorem*}[Theorems \ref{thm:st_vs_iff_proper} and \ref{thm:vs_iff_proper}]
Let $\Ee$ be a stable parabolic bundle, and let $h_{\Ee,nilp}$ (respectively $h_{\Ee,st}$) be the restriction of the Hitchin map to $\Vn$ (resp. $\Vst$). Then, 
$$
\begin{array}{lcl}
\Ee \textnormal{ \it is  (resp.\  strongly) very stable } &\iff&  h_{\Ee,nilp} \,(resp.\  h_{\Ee,st})\textnormal{ \it is finite }\\
&\iff& h_{\Ee,nilp}\,(resp.\  h_{\Ee,st}) \textnormal{ \it is quasi-finite } \\
&\iff&  h_{\Ee,nilp} \,(resp.\  h_{\Ee,st})\textnormal{ \it is proper }\\
&\iff&  \Vn\,\plonge \M(\Pp,\alpha)  \textnormal{ \it is proper}\\
&&(resp.\  \Vst\plonge \M(\Pp,\alpha) \textnormal{ \it is proper}).
\end{array}
$$
\end{theorem*} 
It is worth noting that  finiteness follows from very stability as a corollary of \cite[Lemma 1.3]{Hacen}, of which we provide an alternative elementary proof. Moreover, in the case of nilpotent Higgs fields, properness of $h_{\Ee,nilp}$ is equivalent to properness of $h_{\Ee}:=h_{\Pp,\alpha}|_{\VE}$ (cf. Theorem \ref{thm:vs_iff_proper}). 

An immediate corollary is the fact that  
if $\Ee$ is strongly very stable, then $V_{\Ee,st}$ is a Lagrangian multisection of the restriction  $h_{\Pp,\alpha, st}$ of the Hitchin map (Corollary \ref{cor:lagrangian_multi})  to the strongly parabolic leaf. This was shown in the usual setup in \cite[Corollary 7.3]{FGOP}. A similar result holds for $h_{\Ee,nilp}$, except that in that case it will be Lagrangian in a symplectic leaf possibly of lower dimension than the regular nilpotent (Corollary \ref{cor:multisection_nilp}). Finally, Lemma \ref{lm:wobbly} characterises the strong wobbly locus in terms of intersections of irreducible components of the nilpotent cone in the moduli space of strongly parabolic Higgs bundles with the moduli space of parabolic bundles.

Section \ref{sec:W=S} addresses the identification of wobbly and shaky bundles. In this setup, wobbliness can be defined with respect to the full moduli space or  different symplectic nilpotent leaves. We consider strongly wobbly parabolic bundles (that is, the complement of the strongly very stable locus in the moduli space) and wobbly parabolic bundles (namely, the complement to the very stable locus). Regarding the corresponding shaky loci, let $\hat{r}$ be  a resolution of $r$ obtained by blowup along the locus $\Un$ of stable parabolic Higgs bundles with underlying unstable parabolic bundle (say, when the moduli space is regular). Let  $\Ex$ be the exceptional divisor, and let $\Sh=\hat{r}(\Ex)$. By construction, $\hat{r}$ restricted to $\Mn$ and $\Ms$ is again a succession of blowups, and we denote by $\Shn$ and $\Sh_{st}$ the images of the exceptional divisors. Then:
 \begin{theorem*}[Theorem \ref{thm:W=S}]
 There is an equality $\Wo=\Shn=\Sh$, where $\Wo\subset\N(\Pp,\alpha)$ denotes the wobbly locus.
 
 Similarly, the strong wobbly locus $\Wo_{st}$ satisfies $\Wo_{st}=\Sh_{st}$.
 \end{theorem*}
 This was proven in  \cite{DPHecke} for rank two parabolic bundles on the projective line minus five points. In this case, strongly parabolic bundles and nilpotent parabolic bundles coincide, so there is just one wobbly locus to consider. Let me point out that  these authors work on the non smooth case, so shakiness is defined with respect to a specific point of the Hitchin base. They prove that the definition is independent of this choice.
 
	In the vector bundle setup, shakiness of wobbly bundles has been known since \cite{L, BNR}, and follows from similar arguments to those of \cite[\S 4.1]{BNR}. The proof of this fact can be adapted to the parabolic setup, with a bit of extra work for nilpotent parabolic  Higgs bundles given the existence of different leaves and images of the Hitchin map. This is achieved in Theorem \ref{thm:vs_iff_proper} (see the proof of Theorem \ref{thm:W=S}).

 Finally, Section \ref{sec:classical} explains how our results yield analogous statements in the non-parabolic case. Motivated by Brill--Noether theory, we introduce the notion of $(P,D)$-strong very stability on vector bundles for a parabolic group $P\leq \GL_{n}(\CC)$ and a reduced divisor $D$. For $P=\GL(n,\CC)$ and arbitrary non empty $D$, this recovers the notion of very stability defined by Drinfeld \cite{L}. Using the results in Sections \ref{sec:criteria} and \ref{sec:W=S} we recover the criteria from \cite{PP}, stability of very stable bundles \cite{L} and prove the equality of the usual wobbly and shaky bundles.
 
\subsubsection*{Acknowledgments}
The author would like to thank D. Alfaya, R. Donagi, P.B. Gothen, T. Hausel, J. Martens,   T. Pantev, C. Pauly, C. Simpson,  O. Villamayor, A. Zamora and H. Zelaci for useful discussions on these and related questions.  Many thanks to J.C. D\'iaz-Ramos, M. Dom\'inguez-V\'azquez and E. Garc\'ia-R\'io for their hospitality during the author's stay at Universidade de Santiago de Compostela in 2020, where a part of this article was written. Finally, thanks to the anonymous referees for their detailed comments.

\bigskip

\section{Parabolic Higgs bundles}\label{sec:prel}
We gather in this section some results that will be useful.

\subsection{Parabolic bundles and parabolic Higgs bundles}
Let $D\subset X$ be a reduced divisor of degree $d$.  
Let $\mathcal{P}:=\{P_x:x\in D\}$ be a set of parabolic subgroups of $\GL(n,\C)$ indexed by $D$. Denote by $L_x<P_x$ the Levi subgroup, and let
$\mathfrak{l}_x:=\mathrm{Lie}(L_x)<\mathfrak{p}_x:=\mathrm{Lie}(P_x)$ be the respective Lie algebras. We will denote by $\Oo=\prod_{x\in D}\Oo_x$ an orbit in $\mathfrak{l}:=\prod_{x\in D}\mathfrak{l}_x$ under conjugation by $L:=\prod_{x\in D} L_x$ (so that $\Oo_x$ is a adjoint orbit in $\mathfrak{l}_x$). We define $\mathfrak{n}:=\prod_{x\in D}\mathfrak{n}_x<\mathfrak{p}:=\prod_{x\in D}\mathfrak{p}_x$ the nilpotent subalgebra.

A quasi-parabolic bundle of flag type $\Pp$ (note that the rank of the bundle is determined by $\Pp$, as the subgroups are taken inside a fixed $\GL(n,\CC)$) is a pair $\Ee=(E,\sigma)$ where $E$ is a rank $n$ vector bundle and $\sigma:E|_D\longrightarrow \prod_{x\in D} \GL(n,\C)/P_x$ is a reduction of the structure group to the given parabolic at each point of $D$. Without loss of generality, we will assume that $P_x$ is associated to the partition $n=\sum_{i=1}^{r_x}m_{x,i}$, and we will identify $\sigma$ with a collection of flags $\{E_{x,\bullet}:0=E_{x,0}\subset E_{x,1}\subset\dots\subset E_{x,r_x}=E_x\,:\, x\in D\}$. Let $\alpha_x=(\alpha_{x,1},\dots,\alpha_{x,r_x})$ be a increasing  $r_x$-uple of negative real numbers, that we may assume to be contained in $(-1,0]$ \cite{SNonComp}. Moreover, since every moduli space is isomorphic to one with rational weights  \cite[\S 2]{MS}, we may take $\alpha_{x,i}\in \mathbb{Q}$ for all $x\in D,\,i=1,\dots,r_x$. 
A parabolic vector bundle of type $(\Pp,\alpha)$ is a quasi-parabolic vector bundle $\Ee$ of type $\Pp$ together with the assignation of the numbers $\alpha_{x,i}$ to each step of the flag $E_{x,i}$ $i=1,\dots, r_x$. The assignment of these tuples to a quasi-parabolic bundle $\Ee$ defines a parabolic structure on it. The parabolic bundle thus defined can be assigned an invariant called {parabolic degree} 
$$\pardeg(\Ee):=\deg(E)-\sum_{x\in D}\sum_{i=1}^{r_x}\alpha_{x,i}m_{x,i}$$
where $m_{x,i}=\dim E_{x,i}/E_{x,i-1}$ are called the {multiplicities}. The {parabolic slope} of $\Ee$ is the invariant
$$
\parmu(\Ee):=\frac{\pardeg(\Ee)}{\rk E}.
$$
A parabolic vector bundle $\Ff$  of type $(\Pp',\alpha')$ is a parabolic sub-bundle of a parabolic bundle $\Ee$  of type $(\Pp,\alpha)$ if $F$ is a subbundle of $E$, $P_x'\leq P_x$ and $\alpha_{x,i}'\geq\alpha_{x_i}$. The parabolic bundle $\Ee$ is $\alpha$-semistable if for any proper parabolic sub-bundle $\Ff\subset \Ee$
$$
\parmu(\Ff)\leq \parmu(\Ee).
$$
Clearly, to elucidate whether a parabolic bundle is semistable, it is enough to consider quasi-parabolic sub-bundles $\Ff$ with minimal compatible weights. 

 The moduli space $\N(\Pp,\alpha)$  of parabolic bundles of fixed degree and type is a projective variety  \cite[Theorem 4.1.]{MS} (in particular, it is separated), 
 possibly singular, with closed points parametrising S-equivalence classes of semistable parabolic bundles (or equivalently, isomorphism classes of polystable parabolic bundles). Its dimension is \cite[Theorem 5.3]{MS}
 \begin{equation}\label{eq:dim_N}
   \dim \N(\Pp,\alpha)=(g-1)n^2+1+\frac{1}{2}\left(\sum_{x\in D}n^2-\sum_{i=1}^{r_x}m_{x,i}^2\right).
 \end{equation}{}

A parabolic Higgs bundle of type $(\Pp,\alpha)$ is a pair $(\Ee,\varphi)$ where $\Ee$ is a parabolic vector bundle of type $(\Pp,\alpha)$ and
$\varphi\in H^0(X,\End(E)\otimes K(D))$ is a Higgs field preserving the quasi-parabolic structure (namely, the flag at the prescribed points). Endomorphisms satisfying this condition are called parabolic, and the corresponding sheaf is denoted by $\PEnd(\Ee)$. The subsheaf $\SPEnd(\Ee)\subset\PEnd(\Ee)$ of endomorphisms that induce the zero Higgs field on the graded object is called the sheaf of strongly parabolic endomorphisms. Note that there are exact sequences
\begin{equation}\label{eq:PEnd}
0\longrightarrow\PEnd(\Ee)\longrightarrow\End(E)\longrightarrow\mathfrak{gl}(n,\CC)^{|D|} /\mathfrak{p}\longrightarrow 0,
\end{equation} 
where the rightmost arrow are given by restriction to the divisor $D$, and
\begin{equation}\label{eq:SPEnd}
0\longrightarrow\SPEnd(\Ee)\longrightarrow\PEnd(\Ee)\longrightarrow\mathfrak{p}/\mathfrak{n}\cong\mathfrak{l}\longrightarrow 0,
\end{equation}
where the rightmost arrow is given by restriction to the divisor $D$, followed by the surjection $\mathfrak{p}\projects\mathfrak{l}$. 

Thus, a parabolic Higgs bundle is a pair $(\Ee,\varphi)$ where $\Ee$ is a parabolic vector bundle and
$\varphi\in H^0(X,\PEnd(\Ee)\otimes K(D))$. Semistability is defined similarly to the parabolic bundle case, except that the subsheaves taken into account are those preserved by the Higgs field. The corresponding moduli space $\M(\Pp,\alpha)$ is a quasi-projective variety \cite[Theorem 2.4.8]{Yoko} of dimension \cite[Theorem 5.2]{YokoInf} 
$$
\dim\M(\Pp,\alpha)=(2g-2+d)n^2+1
$$
where $d=\deg D$.  See also \cite{Thaddeus,MichaHilbert} for strongly parabolic Higgs bundles.

\subsection{Poisson structure of the moduli space of Higgs bundles}\label{sec:Poisson} The moduli space of Higgs bundles is Poisson with hyperkaehler leaves. Underlying this rich geometry is the nonabelian Hodge correspondence, a diffeomorphism between three moduli spaces: the Dolbeault (or Higgs) moduli space, the De Rham moduli space of meromorphic connections on $(X,D)$ and the Betti moduli space of representations of the punctured surface fundamental group. The nonabelian Hodge correspondence is due to the work of numerous authors of which we stress in here \cite{SNonComp}.

 The symplectic geometry of the Betti moduli space has mainly been studied by Boalch \cite{BoalchQuasi}. From the Dolbeault point of view, Bottacin \cite{Bottacin} and Markman \cite{Markman} studied the symplectic geometry for meromorphic Higgs bundles (see Section \ref{sec:classical} for more details). 

Following \cite[\S3.2.4]{LM}, consider the complex
$$
C_{\bullet}: \PEnd(\Ee)\stackrel{\ad(\varphi)}{\longrightarrow }\PEnd(\Ee)\otimes K(D)
$$
whose dual complex reads
$$
C_{\bullet}^*: \SPEnd(\Ee)\stackrel{\ad(\varphi)}{\longrightarrow }\SPEnd(\Ee)\otimes K(D)
$$
by duality of $\PEnd(\Ee)$ and $\SPEnd(\Ee)(D)$.

If $(\Ee,\varphi)$ is stable, then the Zariski tangent space of $\M(\Pp,\alpha)$ at $(\Ee,\varphi)$ is identified with the hypercohomology group
\[T_{(\Ee,\varphi)}\M(\Pp,\alpha)=\HH^1(C_{\bullet}),\qquad T_{(\Ee,\varphi)}^*\M(\Pp,\alpha)=\HH^1(C_{\bullet}^*),\]
and the Poisson structure of $\M(\Pp,\alpha)$ is given by the natural map
\begin{equation}\label{eq:poisson_map}
\eta_{{(\Ee,\varphi)}}:T_{(\Ee,\varphi)}^*\M(\Pp,\alpha)\longrightarrow T_{(\Ee,\varphi)}\M(\Pp,\alpha)
\end{equation}
defined by Serre duality. 

More precisely,  the  space $\HH^1(C_\bullet)$ of infinitesimal deformations of any parabolic Higgs bundle sits into an exact sequence
\begin{equation}\label{eq:defs}
0\longrightarrow\HH^0(C_\bullet)\longrightarrow H^0(\PEnd(\Ee))\longrightarrow H^0(\PEnd(\Ee)\otimes K(D))\longrightarrow
\end{equation}
\[
\phantom{(2.5)0\longrightarrow}\HH^1(C_\bullet)\longrightarrow H^1(\PEnd(\Ee))\longrightarrow H^1(\PEnd(\Ee)\otimes K(D))\longrightarrow
\]
\[
\phantom{(2.5) |0\longrightarrow}\HH^2(C_\bullet)\longrightarrow 0\phantom{ H^1(\PEnd(\Ee))\longrightarrow H(\PEnd(\Ee)\otimes K(D))\longrightarrow}.
\]
The term $H^1(\PEnd(\Ee))$ is the space of infinitesimal deformations of the underlying parabolic bundle, which is Serre dual to the space of strongly parabolic Higgs fields. This defines the Poisson map \eqref{eq:poisson_map}. In particular, if $\Ee$ is stable
\begin{equation}\label{eq:T_N}
T_{\Ee}\N(\Pp,\alpha)=H^1(\PEnd(\Ee)).
\end{equation}

The symplectic leaves are determined by the orbits of $\mathrm{gr}(\varphi_x)\in \End(\mathrm{gr}(E_{\bullet,x}))\otimes K(D)_x$ for all $x\in D$, namely, the twisted endomorphism induced by the Higgs field on the graded vector space $\mathrm{gr}(E_{x,\bullet}):=\oplus_{i=1}^{r_x}E_{x,i}/E_{x,i-1}$ at each of the punctures $x\in D$. Let $\mathfrak{p}:=\prod_{x\in D}\mathfrak{p}_x$ and let $\mathfrak{l}:=\prod_{x\in D}\mathfrak{l}_x$ be its Levi quotient (so that $\mathrm{gr}(\varphi_x)\in\mathfrak{l}_x$). There is a short exact sequence of complexes
$$
\xymatrix{
	0\ar[r]&\SPEnd(\Ee)\ar[r]\ar[d]_{-[\varphi, \cdot]}&\PEnd(\Ee)\ar[r]\ar[d]_{[\varphi,\cdot ]}&\mathfrak{l}\ar[r]\ar[d]_{[\mathrm{gr}(\varphi), \cdot]}&0\\	
	0\ar[r]&\SPEnd(\Ee)\otimes K(D)\ar[r]&\PEnd(\Ee)\otimes K(D)\ar[r]&\mathfrak{l}\otimes K(D)|_D\ar[r]&0.
}
$$
Note that in particular, when $\varphi$ is strongly parabolic,  the left and central columns are the defomation complexes of the moduli space of strongly parabolic Higgs bundles $\Ms$ and $\M(\Pp,\alpha)$ respectively. Thus, taking hypercohomology of the columns, we get a long exact sequence:
\begin{equation}\label{eq:LES_hyper}
0\longrightarrow \CC\longrightarrow\Ker([\mathrm{gr}(\varphi),\cdot])\longrightarrow T_{(\Ee,\varphi)}^*\M(\Pp,\alpha)\stackrel{\eta}{\longrightarrow} T_{(\Ee,\varphi)}\M(\Pp,\alpha)\longrightarrow0.
\end{equation}
Therefore the symplectic leaves (that is, the submanifolds along which the Poisson bracket restricts to a symplectic form) are in correspondence with adjoint orbits  $\Oo$ inside $\mathfrak{l}$. Indeed, the rank of $\eta$ is constant if and only if so is the rank of the kernel of $\mathrm{gr}(\varphi)$. Moreover, by fixing $\Oo$, or equivalently, by fixing its centraliser subalgebra  $\mathfrak{z}\subset\mathfrak{l}$, one obtains the integrable distribution $T^*\M(\Pp,\alpha)/\ad(\mathfrak{z})$ of $T_{(\Ee,\varphi)}\M(\Pp,\alpha)$, whose underlying manifold is precisely the the symplectic leaf $\MO$  associated to the orbit $\Oo$.
 We will use the following special notation for the main leaves under study: the regular nilpotent symplectic leaf ${\M(\Pp,\alpha)}_{nilp, reg}$ (corresponding to the regular nilpotent orbit),  and the strongly parabolic symplectic leaf $\M(\Pp,\alpha)_{st}$ (corresponding to the zero orbit). Likewise, we define  $\Mn:=\overline{\M(\Pp,\alpha)}_{nilp, reg}$ to be the closure of the regular nilpotent leaf. Note that $\M(\Pp,\alpha)_{st}\subset \Mn$ is the only closed stratum, contained in the closure of all the other nilpotent leaves.

Similarly to the case of vector bundles, denoting by $\N(\Pp,\alpha)^s\subset \N(\Pp,\alpha)$ the locus of stable points, the space $T^*\N(\Pp,\alpha)^s$ is a dense open subset of $\M(\Pp,\alpha)_{st}$ \cite[Remark 5.1]{YokoInf}. In particular, for $\Ee\in \N(\Pp,\alpha)$, the space $H^1(X,\PEnd(\Ee))=H^0(X,\SPEnd(\Ee)\otimes K(D))^*$ of infinitesimal deformations  of $\Ee$ matches the Zariski tangent space whenever $\Ee$ is stable.

The next lemma is a straightforward generalisation of \cite[Lemma 2.1]{LM}, so we omit the proof.
\begin{lemma}\label{lm:no_nilp_par_end}
	Let $\mathfrak{s}\subset\mathfrak{l}$ be a nilpotent subalgebra. Let  $\Ee$ be an $\alpha$-stable parabolic bundle of flag-type $\Pp$, and $\PEnd(\Ee)_{\mathfrak{s}}\subset\PEnd(\Ee)$ be the subsheaf defined by the exat sequence
	\[
	0\longrightarrow \PEnd(\Ee)_{\mathfrak{s}}\longrightarrow\PEnd(\Ee)\longrightarrow\mathfrak{l}/\mathfrak{s}\longrightarrow 0,
	\]
	where $\PEnd(\Ee)\longrightarrow\mathfrak{l}/\mathfrak{s}$ is the composition of the restriction $\PEnd(\Ee)\longrightarrow\mathfrak{p}$ $\varphi\mapsto\varphi|_D$, and the quotients $\mathfrak{p}\projects\mathfrak{l}\projects\mathfrak{l}/\mathfrak{s}$. Then $H^0(\PEnd(\Ee)_{\mathfrak{s}})=0$. 
	
	Moreover, $H^0(\PEnd(\Ee))\cong\CC$.
\end{lemma}

\subsection{The Hitchin map} The Hitchin map 
\begin{equation}\label{eq:Hitchin_map}
h_{\Pp,\alpha}\,:\, \M(\Pp,\alpha)\longrightarrow \H_{D}:=\bigoplus H^0(X,K^i(iD))
\end{equation}
assigns to each Higgs bundle the characteristic polynomial of the Higgs field.
This is a projective map \cite[Corollary 5.12]{Yoko}. In particular, it is proper and of finite type.

Logares--Martens studied the complete integrability of the restriction of \eqref{eq:Hitchin_map} to the generic symplectic leaves. These correspond to semisimple regular orbits (see \cite[\S 3.2.4]{LM}). However, many of their arguments extend verbatim to the non-semisimple regular orbits, and we will mention this when needed. Baraglia--Kamgarpour extended the study of the integrable system to strongly parabolic bundles \cite{BK}. See also \cite{SWW} for proofs in arbitrary characteristic.

\subsection{$\CC^\times$-action}\label{sec:C_action} By \cite[\S 8]{SNonComp}, there is a action of $\CC^\times$ on $\M(\Pp,\alpha)$ given by
multiplying the Higgs field by scalars
$$ \lambda \cdot (\Ee, \varphi) \mapsto (\Ee, \lambda \cdot \varphi).$$
The Hitchin map is $\CC^\times$-equivariant for this action and a suitable weighted action on $\H_D$ (with weights given by the degrees of the generators of the ring of invariants $\CC[\mathfrak{gl}_n(\CC)]^{\GL(n,\CC)}$). By properness of the Hitchin map, the limits
$$
\lim_{\lambda\to 0}\lambda\cdot(\Ee,\varphi)
$$
exist and belong to the nilpotent cone $h_{\Pp,\alpha}^{-1}(0)$, as so do the limits 
\[
\lim_{t\to\infty} t\cdot (\Ee,\varphi)
\]
when $\varphi$ is nilpotent. Both these kinds of limits are fixed under the $\CC^\times$-action. Fixed points under the $\CC^\times$-action were characterised \cite[Theorem 8]{SNonComp}. Before discussing this, we start with a definition.

\begin{definition}{\cite[Proof of Theorem 8]{SNonComp}}
   A system of Hodge bundles is a parabolic Higgs bundle $(\Ee,\varphi)$ such that $E=\bigoplus_{p=1}^s E^p$ and $\varphi:E^p\longrightarrow E^{p-1}\otimes K(D)$. 
\end{definition}
It is proven in \cite[Proof of Theorem 8]{SNonComp}, that fixed points are precisely systems of Hodge bundles. 
 
\begin{remark}\label{rk:Hodge_bundle_str_par}
	We note that if  $(\Ee,\varphi)$ is  fixed for the $\CC^\times$-action, there exists a refinement of the flags such that $(\Ee,\varphi)$ becomes strongly parabolic with respect to the latter. This is a general fact for nilpotent quasi-parabolic bundles (see the proof of Proposition \ref{prop:image_V_nilp}), but $\CC^\times$-fixed points have a preferred refinement preserving the structure of a system of Hodge bundles. 
 
 To see this in the special case of fixed points, note that, as explained in \cite[Theorem 8]{SNonComp}, fixed points for the $\CC^\times$-action have underlying parabolic bundle of the form $\Ee=(\bigoplus_{p=1}^s E^p,\sigma)$, with $E^p$ the eigenspace of the automorphism of $E$ swapping $\varphi$ and $t\varphi$ \cite[Lemma 4.1]{SimpsonVariations} (where we have used that if $\Ee=(E,\varphi)$ is a fixed point, then so is $E$). In turn, $\sigma$ preserves the structure of a system of Hodge bundles (i.e., letting $m_p=\rk(E^p)$, $\ol{m}=(m_1,\dots,m_s)$ be the corresponding partition and $P_{\ol{m}}$ be the associated parabolic, then  $P_{m_p}<P_x$, so that  $\sigma$   induces a $\GL(m_p,\CC)$-equivariant morphism $\sigma: E^p_x\to \GL(m_p,\CC)/P_{p,x}$). The fact that $\varphi:E^p\longrightarrow E^{p-1}\otimes K(D)$ follows from the definition of the eigenspaces.  
\end{remark}{}
We note that since nilpotency along $D$ is preserved by the $\CC^\times$-action, $\M(\Pp,\alpha)_{nilp}$ is  $\CC^\times$-invariant. On the other hand, $h_{\Pp,\alpha}^{-1}(0)\subset \M(\Pp,\alpha)_{nilp}$, so that studying the nilpotent cone for $\M(\Pp,\alpha)$ amounts to studying it for $\M(\Pp,\alpha)_{nilp}$. Remark \ref{rk:Hodge_bundle_str_par} implies that the strongly parabolic leaf contains all the information needed to recover the fixed points.
\subsection{Parabolic (Higgs) bundles as (Higgs) bundles on orbicurves}\label{sec:rational_weights} For rational weights $\alpha$ the stack of parabolic bundles (respectively strongly parabolic Higgs bundles) is isomorphic to a stack of bundles (respectively Higgs bundles) on orbicurves \cite{BiswasOrbi,Borne} whose orbifold structure is determined by the weights. In particular, the results of \cite{L} on the nilpotent cone of the stack of Higgs bundles on a smooth curve extend to this setup  \cite[\S 5.4]{DPHecke}. Given the fragmentation of the literature on this topic, we will hereby include a detailed proof of some of these results (see Lemma \ref{lm:vs_is_stable}). 

Let $\hat{X}_{D,\alpha}$ be the root stack (or orbifold) obtained from $(X,D,\alpha)$ (recall that we assume weights to be rational) \cite{Cadman}. It is a smooth Deligne--Mumford stack. Consider the categories $\mathbf{Par}(X,D,\alpha)$ (resp. $\mathbf{Vec}(\hat{X}_{D,\alpha})$ of vector bundles on $\hat{X}_{D,\alpha}$. These are tensorial categories \cite{MY,Borne}, which are equivalent under a twist by a line bundle composed with  pushforward via the map $\hat{X}_{D,\alpha}\longrightarrow X$ \cite[Theorem 3.13]{Borne}. 

\subsection{Universal bundles and locally universal bundles}\label{sec:univ_bundles}
Yokogawa proved that $\M(\Pp,\alpha)$ is an open subset of a good quotient of a projective scheme by $\mathrm{PGL}(N)$  \cite[Theorem 4.6]{Yoko}, with stable locus  $\M(\Pp,\alpha)^s$ contained in the corresponding geometric quotient, giving the smooth points \cite[Theorem 5.2]{YokoInf}. As an open subset $\M(\Pp,\alpha)^s$ containts the moduli space $\M(\Pp,\alpha)^0$ of parabolic Higgs bundles with underlying stable parabolic bundle. The inclusion is strict by \cite[Claim 3.2 (i)]{Boden-Yoko} (see also Remark \ref{rk:Un}): even for non generic weights there are more stable pairs than those in $\M(\Pp,\alpha)^0$. By \cite[Remark 5.1]{YokoInf}, \'etale locally over $\M(\Pp,\alpha)^0$  there is a universal bundle. The action of the center of $\GL(N)$ is non trivial \cite[page 464]{Boden-Yoko}, and therefore the bundle does in general not descend. It only descends to $\M(\Pp,\alpha)^0$ when the parabolic weights are generic \cite[Proposition 3.2]{Boden-Yoko}.  However, regardless of the weights, points of $\M(\Pp,\alpha)^{0}$ always have an \'etale neighbourhood admitting a universal bundle \cite[page 16]{Bhosle}, \cite[page 464]{Boden-Yoko}. The same holds for $\N(\Pp,\alpha)^s$.

\section{Parabolic bundles and very stability}\label{sec:criteria}

\subsection{Criteria for very stability} In this section we prove a criterion for very stability of parabolic bundles via the Hitchin map. 

Let us start by recalling \cite[Lemma 1.3]{Hacen}. We likewise provide an alternative simple proof via toric geometry that was hinted to us by an anonymous referee.
\begin{lemma}{\cite[Lemma 1.3]{Hacen}}\label{lm:Hacen}
Let $ f=(f_1,\dots, f_n):\mathbb{A}^m\longrightarrow\mathbb{A}^n$ be a morphism given by homogeneous polynomials. Then, if $f^{-1}(0)=0$, $f$ is finite.
\end{lemma}
\begin{proof}
Note that since $f^{-1}(0)=0$, the morphism $f$ extends to a morphism
\[
\ol{f}:\PP^m\longrightarrow\PP(d_1,\dots,d_n,1)\qquad [x_1:\dots:x_m:y]\mapsto[y\cdot f_1(\ol{x}):\dots:y\cdot f_n(\ol{x}):y]
\]
of weighted projective spaces, where $d_i=\deg(f_i)$, which is moreover toric. Indeed, by construction $\ol{f}$ restricts to a torus homomorphism $(\CC^\times)^m\longrightarrow (\CC^\times)^n$ for a suitable group structure on the target torus taking weights into consideration. By \cite[Theorem 3.4.7.]{CLS}, the map $\ol{f}$ is proper, and by upper semicontinuity of the dimension, it is also finite. In particular $\dim\Im(\ol{f})=m\leq n$ (again, by upper semicontinuity of the dimension). Therefore $f:\AA^m\longrightarrow\AA^n$ is a quasi finite morphism. We claim that it is also proper. To see this, consider valued ring $R$ with quotient field $k$, and a commutative diagram
\[
\xymatrix{\Spec(k)\ar[d]\ar[r]&\AA^m\ar[d]\ar[r]&\PP^m\ar[d]\\
\Spec(R)\ar[r]\ar@{-->}[ru]^{l_1}\ar@{-->}[rru]^{l_2}&\AA^n\ar[r]&\PP^n.}
\]
Now, the only way $l_1$ may not exist is if the image of the non generic point $p$ is $0$. Now, by the valuative criterion for properness, the arrow $l_2$ exists and it must be $l_2(p)=\infty$. Therefore, $\Im(l_2)\subset\AA^m$ and $l_1$ also exists.
\end{proof}
Given a quasi-parabolic bundle $\Ee$, we denote $\V_\Ee:=H^0(X,\mathrm{ParEnd}(\Ee)\otimes K(D))$. For each adjoint orbit $\Oo=\prod_{x\in D}\Oo_x$  of $\prod_{x\in D}\mathfrak{l}_x$ (where  $\mathfrak{l}_x<\mathfrak{p}_x$ is the Levi subalgebra), denote by $\V_{\Ee,\Oo}$ the subset of $\VE$ corresponding to Higgs fields with $\gr(\varphi_x)\in\Oo_x$. When $\Oo=0$, we will denote $\VO$ by $\Vst=H^0(X,\SPEnd(\Ee)\otimes K(D))$, and if $\Oo=\Oo_{rn}$ is the regular nilpotent orbit, then the closure of $\VO$ in $\VE$ is denoted by  $\Vn$. This is the set of nilpotent Higgs fields.
\begin{remark}\label{rk:VO_as_intersection}
If $\Ee$ is $\alpha$-stable, then $\VE\subset\M(\Pp,\alpha)$, $\VO=\VE\cap\MO$, and $\Vn=\VE\cap\Mn$, but the definition of these objects is independent of the stability parameter. 
\end{remark}
\begin{definition}
 A quasi-parabolic bundle $\mathcal{E}=(E,\{E_{x,\bullet}\,:\,x\in D\})$ is called very stable if it has no non-zero nilpotent parabolic Higgs field $\varphi\in\VE$. It is called {strongly very stable} if and only if it has no non-zero strongly parabolic nilpotent Higgs field. 
 
 An $\alpha$-stable bundle that is not (strongly) very stable is called (strongly) wobbly. Let $\Wo\subset\N(\Pp,\alpha)$ (respectively $\Wo_{st}\subset\N(\Pp,\alpha)$) denote the wobbly locus (resp. the strong wobbly locus). 
\end{definition}
\begin{remark}\label{rk:vs_indep_weight}
\begin{enumerate}
    \item\label{it:vs_depends_on_qp} Although the definition of (strong) very stability only depends on the quasi-parabolic structure,  for every assignment of a parabolic structure $(\Pp,\alpha)$, very stable bundles are $\alpha$-stable (see Lemma \ref{lm:vs_is_stable}  and Remark \ref{rk:wobbly_alpha}).
    \item\label{it:vs=nvs}  Note that $\mathcal{E}$ is very stable if and only if there is no nilpotent Higgs field in $\Vn\setminus{\{0\}}$.
    \item\label{it:vs_is_svs} By definition, very stable parabolic bundles are strongly very stable.
\end{enumerate}{}
\end{remark}
When the weights are rational, as observed in Section \ref{sec:rational_weights}, the tools developed in \cite{L} are available to analyse the moduli space of strongly nilpotent Higgs bundles. This gives the following lemma.
\begin{lemma}
	\label{lm:vs_is_stable}
	Let $\Ee$ be parabolic of type $(\Pp,\alpha)$. If it is strongly very stable, then it is $\alpha$-stable. In particular, if it is very stable, then it is $\alpha$-stable.  
\end{lemma}
\begin{proof}
	By assigning rational weights $\alpha$, strongly very stable parabolic bundles can be seen as very stable bundles on an orbifold curve whose stability can be proven by elementary means just as in \cite{L}, using the Grothendieck--Riemann--Roch theorem for root stacks (see \cite[Theorem 3.10]{DV} for an explicit formulation for curves based on \cite{Edidin, Toen}). 
 
 Indeed, if $\Ee$ were not stable, then there would exist a semistable parabolic sub-bundle $\Ff\subset \Ee$ with $\parmu(\Ff)\geq \parmu(\Ee)$ such that the quotient  $\Qq:=\Ee/\Ff$ is also semistable. Let $\beta=(\beta_{x,0},\dots,\beta_{x,r_x})_{x\in D}$ (resp. $\gamma=(\gamma_{x,0},\dots,\gamma_{x,r_x})$  denote the weights for $\Ff$ (resp. $\Qq$). Note that by allowing the lengths to be the same as those for $\Ee$, we are allowing for $\beta_{x,j}=\beta_{x,j+1}$, but this implies that $F_{x,j}=F_{x,j+1}$. A similar statement holds for $\Qq$. Then, $H^0(\mathrm{PHom}(\Ff,\Qq))\geq 0$\footnote{Note that equality holds if $\parmu(\Ff)>\parmu(\Qq)$, as otherwise semistability of both $\Ff$ and $\Qq$ cannot hold.}, so by Serre duality \cite[Prop 3.7]{YokoInf} for parabolic sheaves $H^1(\mathrm{SPHom}(\Qq,\Ff)K(D))\geq0$. Here $\mathrm{PHom}(\Ff,\Qq)$ (respectively 
    $\mathrm{SPHom}(\Qq,\Ff)$) is the sheaf of morphisms inducing $Q_{x,j}\longrightarrow F_{x,k-1}$ when $\gamma_{x,j}<\beta_{x,k}$ (resp. $\gamma_{j}\leq\beta_{k}$). On the other hand, by \cite[Theorem 3.10]{DV}
    \begin{equation}
        \label{eq:RR}-\chi(\mathrm{PHom}(\Ff,\Qq))=
            \end{equation}
\[-\pardeg(\mathrm{PHom}(\Ff,\Qq))+(g-1)\rk(\mathrm{PHom}(\Qq,\Ff))-\sum_{x\in D}\sum_{j=1}^{\tilde{r}_{x}}\tilde{\alpha}_{x,j}\tilde{m}_{x,j}.    \]
    In the above, the numbers $\tilde{r}_{x}$, $\tilde{\alpha}_{x}$  and $\tilde{m}_{x}$ are the corresponding data associated to the induced filtration for $\mathrm{PHom}(\Ff,\Qq)$, which are the homomorphisms in the category of parabolic sheaves \cite{Borne, YokoInf}. We note that our expression of the Euler characteristic \eqref{eq:RR} varies slightly from the one in \cite[Theorem 3.10]{DV}, the reason being that in {\it loc.cit.} weights are taken to be positive, while we consider negative weights. With these considerations, formula \eqref{eq:RR} follows directly from \cite[Theorem 3.10]{DV} by the equivalence of categories from \cite[Theorem 3.13]{Borne}.

    The result is straightforward from \eqref{eq:RR} by two remarks. Firstly, the following inequalities hold:
    \[
h^0(\mathrm{SPHom}(\Qq,\Ff)\otimes K(D))\geq-\chi(\mathrm{PHom}(\Ff,\Qq))\geq (g-1)\rk(\Qq)\rk(\Ff)>0.
    \]
 where the first inequality follows from Serre duality and $h^0(\mathrm{PHom}(\Ff,\Qq))\geq 0$. Secondly,  the morphism
    \[
\Ee\twoheadrightarrow \Qq\stackrel{\psi}{\longrightarrow}\Ff\otimes K(D)\hookrightarrow \Ee\otimes K(D)
    \]
    is strongly parabolic for any $\psi\in H^0(\mathrm{SPHom}(\Qq,\Ff)\otimes K(D))$.
\end{proof}
\begin{remark}\label{rk:SWW}
    In \cite[Theorem 6.14]{SWW}, a proof of the existence of a Zariski open set of the moduli space of parabolic bundles which is contained in the set of strongly very stable parabolic bundles is given using a dimensional argument. Lemma \ref{lm:vs_is_stable} improves this result (provided non emptiness of the very stable locus) by showing that all strongly very stable parabolic bundles are stable. Non emptiness follows from \cite[Theorem 3.1]{L} (see \S \ref{sec:rational_weights}), or, alternatively, from Lemma \ref{lm:wobbly}.
\end{remark}    
\begin{remark}\label{rk:wobbly_alpha}
	Note that the proof of Lemma \ref{lm:vs_is_stable} uses the identification of parabolic bundles with bundles on orbicurves, and this requires the assignment of weights. However, since  the definition of (strong) very stability is independent of $\alpha$ (cf. Remark \ref{rk:vs_indep_weight}), from Lemma \ref{lm:vs_is_stable} we deduce the existence of a common subset to all $\N(\Pp,\alpha)$, namely, the set of very stable bundles $\Ee$. The corresponding wobbly locus, however, depends on the stability parameter $\alpha$. It follows that the moduli spaces of parabolic bundles for different weights provide different compactifications of the very stable locus.
\end{remark}{}
\begin{theorem}\label{thm:st_vs_iff_proper}
	Let $\Ee\in\N(\Pp,\alpha)$ be stable. Let 
	$\Vst$ be the space of strongly parabolic Higgs fields on $\Ee$, and let $h_{\Ee,st}:=h_{\Pp,\alpha}|_{\Vst}$. Then:
	$$
	\begin{array}{lcl}
	\Ee \textnormal{ \it is strongly very stable } &\iff&  h_{\Ee,st} \textnormal{ \it is finite }\\
	&\iff& h_{\Ee,st} \textnormal{ \it is quasi-finite } \\
	&\iff&  h_{\Ee,st} \textnormal{ \it is proper }\\
	&\iff&  \Vst\plonge \M(\Pp,\alpha)  \textnormal{ \it is proper.}
	\end{array}
	$$
\end{theorem}
\begin{proof}
	The fact that very stability implies finiteness follows from affineness of $\Vst$ and Lemma \ref{lm:Hacen}. 
	
	To see the other implications, we need to work some more. Let $\H_{st}\subset \H_D$ be the image of $\Ms$ under $h_{\Pp,\alpha}$. By \cite[Theorem 36]{BK}, $\H_{st}\subset\H_D$ is an affine subspace of dimension $\dim \H_{st}=\dim\N(\Pp,\alpha)$.

Also by \cite[Theorem 5.3]{MS}, $\dim\N(\Pp,\alpha)=\dim{H^0(X,\SPEnd(\Ee)\otimes K(D))}$ whenever $\Ee$ is stable.

Now, finiteness implies properness and quasi-finiteness. Also, by the above discussion $h_{\Ee,st}$ is a map of finite type of affine spaces of the same dimension, hence properness implies quasi-finiteness and hence also finiteness.  

Regarding the equivalence of quasi-finiteness and finiteness, quasi-finiteness implies very stability (and thus finiteness) as the existence of a non-zero nilpotent Higgs field $\varphi$ on $\Ee$ would automatically produce a one dimensional subspace in $h^{-1}_{\Ee,st}(0)$ (this requires the stability hypothesis on $\Ee$ to make sure none of the Higgs fields in the line $\CC^\times\cdot \varphi$ are identified). 

Finally,  the equivalence between properness of $\Vst\plonge \M(\Pp,\alpha)$ and properness of $h_{\Ee,st}$ is a consequence of the valuative criterion for properness (this is exactly the proof of \cite[Proposition 2.2]{PP}).
\end{proof}{}

The following result was observed in \cite[Corollary 7.3]{FGOP}, and extends verbatim to the current context.
\begin{corollary}\label{cor:lagrangian_multi}Let $\M(\Pp,\alpha)_{st}$  denote the strongly parabolic symplectic leaf, and let 
$h_{\Pp,\alpha,st}:=h_{\Pp,\alpha}|_{\M(\Pp,\alpha)_{st}}$. If $\Ee$ is strongly very stable, then $\V_{\Ee,st}$ is a Lagrangian multisection of  $h_{\Pp,\alpha,st}$.
\end{corollary}
\begin{proof}
Since the only deformations along $\Vst$ concern the Higgs field (see Section \ref{sec:Poisson}), $\Vst$ is clearly isotropic. 

By Lemma \ref{lm:vs_is_stable}, $\Ee$ is stable, and so from \eqref{eq:T_N} and Serre duality $T_\Ee\N(\Pp,\alpha)=H^0(X,\SPEnd(\Ee)\otimes K(D))^*$.  Thus, 
$
\dim\Vst=\dim\N(\Pp,\alpha)^s\subset \Ms$.
But $T^*\N(\Pp,\alpha)\subset\Ms$ is dense, so it follows that 
 $\dim\Vst=\dim\N(\Pp,\alpha)=\frac{1}{2}\dim\Ms$, and so $\Vst$ is maximal dimensional.

 Finally, $\Im(h_{\Pp,\alpha,st})$ is affine by \cite[Theorem 36]{BK} of dimension equal to the dimension of $\Vst$. This, together with properness, yields that $h_{\Ee,st}$ is onto, hence the result. 
\end{proof}{}
\begin{remark}
	Note that the stability assumptions in Theorem \ref{thm:st_vs_iff_proper} and Theorem \ref{thm:vs_iff_proper} may  be dropped in one direction by  Lemma \ref{lm:vs_is_stable}.  
\end{remark}{}
Let
\begin{equation}\label{eq:h_V}
h_{\Ee,nilp} : \Vn \longrightarrow \H_D    
\end{equation}{}
be the restriction of the Hitchin map to the vector space $\Vn:=H^0(X,\PEnd(\Ee)\otimes K(D))_{nilp}$ of Higgs fields on $\Ee$ with nilpotent residue. In order to prove the analogue to Theorem \ref{thm:st_vs_iff_proper} in this setup, let us start by two preliminary results about $\Vn$.
\begin{lemma}\label{lm:oVn_minus_Vn}
	Let $\Ee$ be a stable parabolic bundle, and let $(\Ff,\psi)\in\oVn\setminus\Vn$ or $(\Ff,\psi)\in\oVE\setminus\VE$. Then 
	
	\begin{enumerate}
	\item\label{it:F_uns} $\Ff$ is unstable. 
	
	\item\label{it:oVE_vs_oVn} $\oVn\setminus\Vn\neq\emptyset\iff\oVE\setminus\VE\neq\emptyset$.
\end{enumerate}
\end{lemma}
\begin{proof}
	By  \cite[\href{https://stacks.math.columbia.edu/tag/0A24}{Tag 0A24}]{stacks-project}, given $(\Ff,\psi)\in\oVn\setminus\Vn$, there exists a discrete valuation ring $R$ and a morphism 
	$$
	\iota_F:C:=\Spec(R)\longrightarrow \M(\Pp,\alpha)
	$$
	such that the generic point $\Spec(k)$ (where $k$ is the fraction field of $R$) maps to $\Vn$,  while the closed point $o$ goes to $(\Ff,\psi)\in \ol{\V}_{\Ee,nilp}\setminus{\Vn}$. Since $\Mn\subset\M(\Pp,\alpha)$ is closed, then $\oVn\setminus\Vn\subset\M(\Pp,\alpha)\setminus\VE$, namely, it must be $\Ff\neq\Ee$. So it is enough to prove the statement for $(\Ff,\psi)\in\oVE$.
	
	 If $\Ff$ were semistable, composing with $r$ would yield 
	\[r\circ \iota_F:C:=\Spec(R)\longrightarrow \N(\Pp,\alpha)\] non constant, extending the constant map $\Ee:\Spec(k)\longrightarrow \N(\Pp,\alpha)$, thus violating separatedness of $\N(\Pp,\alpha)$ \cite[Theorem 4.1]{MS}. This proves \eqref{it:F_uns}.
	
	For \eqref{it:oVE_vs_oVn}, note that $\oVn\setminus\Vn\subset\oVE\setminus\VE$ (by \eqref{it:F_uns}, or simply by closedness of $\Mn$). So it is enough to prove the converse. 
	
	Let $(\Ff,\psi)\in\oVE\setminus\VE$. Then, by equivariance of the Hitchin map $(\Ff_0,\psi_0):=\lim_{t\to 0}t\cdot (\Ff,\psi)\in h^{-1}_{\Pp,\alpha}(0)\cap\oVE\subset\Mn\cap\oVE=\oVn$.
\end{proof}
\begin{remark}
	The above proof translates essentially verbatim to the simplest non parabolic vector bundle setup, and corrects a mistake in the proof of \cite[Proposition 2.3]{PP}. Indeed, the  \'etale local family considered therein may not exist away from $T^*\N^s(n,d)$. Thanks to T. Hausel for pointing this error out to us.
\end{remark}
The following adapts \cite[Theorem 36]{BK} to general nilpotent Higgs fields.
\begin{proposition}\label{prop:image_V_nilp}
	Let $\Ee$ be a stable parabolic bundle. Then, the image under the Hitchin map $h_{\Pp,\alpha}$ of $\Vn$ is contained  in an affine space $\AA^d$ with $d=\dim\Vn$. 
\end{proposition}
\begin{proof}
	First note, that  $\Vn=\ol{\MO}\cap \VE$ for some nilpotent orbit $\Oo$ such that $\VO:=\MO\cap\VE\neq\emptyset$. Indeed, take a maximal dimensional orbit satisfying $\VO\neq\emptyset$ . If $\Vn\neq\ol{\MO}\cap \VE$, then there exists another nilpotent orbit $\Oo'$  with $\V_{\Ee,\Oo'}\neq\emptyset$ and $\V_{\Ee,\Oo'}\not\subset \ol{\VO}$. By affineness of $\VE$, then $\VO+\V_{\Ee,\Oo'}\subset\Vn$. Now, using the Jordan canonical form, we see that either $\Oo'\subset\ol{\Oo}$ or there is an element in $\VO+\V_{\Ee,\Oo'}$ belonging to a higher dimensional orbit, which contradicts our choice of $\Oo$. By irreducibility of $\Vn$, one concludes that $\Vn=\ol{\MO}\cap \VE$.
	
	 Since $\Vn\cong\mathbb{A}^d$ is irreducible, the closure of $\VO$ inside $\VE$ must be all of $\Vn$. Now,  for some refinement $\Pp'<\Pp$,  the orbit 
	$\Oo$ intersects  $\mathfrak{l}:=\prod_{x\in D}\mathfrak{l}'_x$ (where $\mathfrak{l}'_x$ is the Levi subalgebra of $\mathfrak{p}_x'=Lie(\Pp_x')$) at $0$. 
	 For example, taking $\Pp'$ determined by the iterated kernels of $\varphi_x\in\Oo$, then  $\Pp'$ satisfies the property and moreover, it is maximal for it (namely, whenever $\ol{\Pp}$ is another refinement with $\varphi$ lifting to a strongly nilpotent field preserving a flag of type $\ol{\Pp}$, then $\ol{\Pp}<\Pp'$).  
	 
	 Let $\mathfrak{p}_x'=\mathfrak{l}_x'\oplus\mathfrak{n}_x'$, and let $\mathbf{O}:=\prod_{x\in D}\mathfrak{n}_x'\cap\mathfrak{l}_x$. Define $\PEnd(\Ee)_{\mathbf{O}}\subset\PEnd(\Ee)$ via the following exact sequence: 
	 $$
	 0\longrightarrow\PEnd(\Ee)_{\mathbf{O}}\longrightarrow\PEnd(\Ee)\longrightarrow\mathfrak{l}/\mathbf{O}\longrightarrow0,
	 $$
	where the rightmost arrow is the composition of 
	$\PEnd(\Ee)\projects \mathfrak{p}\projects \mathfrak{l}\projects \mathfrak{l}/\mathbf{O}$.
	 
	 Now, let $\Ee'$ denote the quasi-parabolic bundle  of type $\Pp'$ induced by the existence of some Higgs field in $\Vn$ with residue in $\Oo$. Note that the existence of the exact commutative diagram:
	 $$
	 \xymatrix{
	 	\SPEnd(\Ee')\ar@{^(->}[r]\ar@{^(->}[d]&\PEnd(\Ee')\ar@{^(->}[d]\ar@{->>}[r]&\mathfrak{l}'\ar@{^(->}[d]\\
	 	\PEnd(\Ee)_{\mathbf{O}}\ar@{^(->}[r]&\PEnd(\Ee)\ar@{->>}[d]\ar@{->>}[r]&\mathfrak{l}/\mathbf{O}\ar@{->>}[d]\\
	 	&\mathfrak{p}/\mathfrak{p}'\ar[r]_{\cong}&\mathfrak{l}/\mathfrak{p}'\cap\mathfrak{l}.
	 }
	 $$
	 implies that
	 \begin{equation}\label{eq:Par_O_SPar}
	 	\PEnd(\Ee)_{\mathbf{O}}\cong\SPEnd(\Ee').
	 \end{equation}
	 Note that $\VO\subset H^0(\PEnd(\Ee)_{\mathbf{O}}\otimes K(D))\subset\Vn$, thus, by affiness of the last two subspaces and equality of the dimensions of $\VO$ and $\Vn$,  the second inclusion must be an equality and so $\dim H^0(\PEnd(\Ee)_{\mathbf{O}}\otimes K(D))=\dim\Vn$. 
	 
	 Now, by stablity of $\Ee$ and Lemma \ref{lm:no_nilp_par_end},
we have that $H^0(\PEnd(\Ee)_\mathbf{O})=0=H^0(\SPEnd(\Ee'))$. A simple computation using \eqref{eq:SPEnd} shows that this implies that $\dim H^0(\SPEnd(\Ee')\otimes K(D))=\dim\N(\Pp',\alpha')$ (for a generic  $\alpha'$). Note that this equality holds unconditionally and does not require the existence of parabolic weights that would turn $\Ee'$ into a stable parabolic bundle. It then follows from the above and \eqref{eq:Par_O_SPar} that
\begin{equation}\label{eq:dim_VO}
\dim\Vn=\dim H^0(\SPEnd(\Ee')\otimes K(D))=\dim\N(\Pp',\alpha')
\end{equation}
	 for a suitable $\alpha'$.
	 
	 Let $\Mm(\Pp)$ be the moduli stack of parabolic Higgs bundles on $X$ of flag type $\Pp$. The Hitchin  map 
	 \[
	 h_\Pp:\Mm(\Pp)\longrightarrow\H_D
	 \]is defined in the same way. 
	  Moreover, there is a morphism $\Mm(\mathcal{P}')\longrightarrow \Mm(\mathcal{P})$ making the following diagram commute:
	\begin{equation}\label{eq:Hitchin_maps_neq_pars}
	\xymatrix{
		\Mm(\mathcal{P}')\ar[d]_{p}\ar[dr]^{h_{\mathcal{P}'}}&\\
	\Mm(\mathcal{P})\ar[r]_{h_{\mathcal{P}}}&\H_{D}.
	}
	\end{equation}
 In particular $h_{\Pp,\alpha}(\VO)\subset \Im(h_{\Pp',st})\subset\Im(h_{\Pp,st})$ which is affine of dimension equal to $\dim\VO$ by \cite[Theorem 36]{BK} and \eqref{eq:dim_VO}. 
\end{proof}

\begin{theorem}\label{thm:vs_iff_proper}
Let $\Ee$ be stable, and let $\Vn$ and $h_{\Ee,nilp}$ be as in \eqref{eq:h_V}. Then, 
$$
	\begin{array}{lcl}
	\Ee \textnormal{ \it is  very stable } &\iff&  h_{\Ee,nilp} \textnormal{ \it is finite }\\
	&\iff& h_{\Ee,nilp} \textnormal{ \it is quasi-finite } \\
	&\iff&  h_{\Ee,nilp} \textnormal{ \it is proper }\\
	&\iff&  \Vn\plonge \M(\Pp,\alpha)  \textnormal{ \it is proper}\\
		&\iff&  h_{\Ee} \textnormal{ \it is proper }\\
	&\iff&  \VE\plonge \M(\Pp,\alpha)  \textnormal{ \it is proper.}\\
	\end{array}
	$$
\end{theorem}{}
\begin{proof}
The equivalence between properness of $\Vn$ and $\VE$ follows from   Lemma \ref{lm:oVn_minus_Vn}. The equivalence between properness of $h_{\Ee,nilp}$ (resp.  $h_\Ee$) and properness of $\Vn$ (resp.  $\VE$) is a consequence of the valuative criterion for properness. Thus it is enough to prove all the other equivalences.

The implication that very stability of $\Ee$ implies finiteness of $h_{\Ee,nilp}$ follows from Lemma \ref{lm:Hacen}. 

To prove the remaining equivalences, we note that it is enough to check the equivalence of finiteness and properness of $h_{\Ee,nilp}$. Indeed, quasi-finiteness of  $h_{\Ee,nilp}$ automatically implies very stability of $\Ee$, as the existence of a nilpotent Higgs field implies the existence of a line of nilpotent Higgs fields. Thus,  quasi-finiteness is equivalent to finiteness by the first equivalence. 

Now, since finiteness of $h_{\Ee,nilp}$ implies properness of $h_{\Ee,nilp}$, it is enough to check that properness of $h_{\Ee,nilp}$ implies finiteness of $h_{\Ee,nilp}$, for which given that the Hitchin map is of finite type, it is enough to check that fibers are finite. This follows from Proposition \ref{prop:image_V_nilp}.
\end{proof}
The analogue of Corollary \ref{cor:lagrangian_multi} follows from Propositions \ref{prop:image_V_nilp} and \ref{thm:vs_iff_proper}: 

\begin{corollary}\label{cor:multisection_nilp}
	Let $\Ee$ be very stable, and let $\Oo$ be a nilpotent orbit such that $\VO\neq\emptyset$ is maximal dimensional inside $\Vn$. Then, $\Vn$ is the closure of a Lagrangian multisection of the restriction of  $h_{\Pp,\alpha}$ to $\MO$.
\end{corollary}
\begin{proof} By  Propositions \ref{prop:image_V_nilp} and \ref{thm:vs_iff_proper}, together with isotropicity of $\VO$ (see Section \ref{sec:Poisson}), all we need to check is that:
	$$
	\dim\VO=\frac{1}{2}\dim\MO.
	$$
	By \eqref{eq:LES_hyper}, the rank of $\eta$ at  $(\Ee,\varphi)$ when $\varphi|_D\in\Oo$ is 
	$$
	\mathrm{\rk}(\eta)=\dim\M(\Pp,\alpha)-\dim\mathfrak{z}(\mathrm{gr}(\varphi))+1.$$
	Thus, taking $\mathfrak{l}'$ as in the proof of Proposition \ref{prop:image_V_nilp}, we have that $\Ker[\mathrm{gr}(\varphi),\cdot]\cong\mathfrak{l}'$, hence
	$$\mathrm{\rk}(\eta)=2(g-1)n^2+2+\sum_{x\in D} n^2-\dim\mathfrak{l}'_x	=2(\dim \VO),
	$$ 	
	where the last equality follows from \eqref{eq:dim_N} and \eqref{eq:dim_VO}.
\end{proof}

\subsection{Strongly wobbly bundles and the nilpotent cone}
In this section we describe strong wobbliness in terms of the geometry of the nilpotent cone. 

Let $C_i\subset h_{\Pp,\alpha,st}^{-1}(0)$, $i\in I$ denote the irreducible components of the strongly parabolic nilpotent cone. By \cite[Proposition 3.8]{L} one of them coincides with $\N(\Pp,\alpha)$, say  $C_0=\N(\Pp,\alpha)$, under the map $\Ee\mapsto(\Ee,0)$ for $\Ee\in\N(\Pp,\alpha)$. Let $C_0^s$ denote the interseccion of $C_0$ and the image of the stable locus $\N(\Pp,\alpha)^s\subset \N(\Pp,\alpha)$ under the aforementioned map. For vector bundles, $I$ is given by all possible partitions of the rank and degree of the same length satisfying a numerical condition \cite[Theorem 3.1]{Bozec}. These partitions correspond to ranks and degrees of the graded pieces of fixed point compotents (cf. \S\ref{sec:C_action}). In particular, $C_0^s$ corresponds to the trivial partition $(n,d)$.

\begin{lemma}\label{lm:wobbly}
The strong wobbly locus of $\N(\Pp,\alpha)$ is the intersection $\Wo_{st}:=\bigcup_{i\neq 0}C_i\cap C_0^s$.
\end{lemma}
\begin{proof}
The inclusion $\subset$ is clear, as if $\Ee\in \Wo$, let $\varphi$ be a strongly parabolic Higgs field on $\Ee$, then $(\Ee,\varphi)\in C_i$ for some $i\neq 0$. Hence $(\Ee,0)=\lim_{t\to 0}(E,t\varphi)\in C_i\cap C_0^s$ for some $i\neq 0$.

For the converse, assume $(\Ee,0)\in C_i\cap C_0^s$. Then $(\Ee,0)$ is a boundary point of $C_i^s\cap\M(\Pp,\alpha)^0$ inside $\M(\Pp,\alpha)^0$. So by \cite[Lemma 2.4]{PP}, and stability of $\Ee$, we may find a smooth curve $\psi:Z\longrightarrow C_i^s$ such that $\psi\left(Z\setminus\{z_0\}\right)\subset C_i^s\setminus \partial C_i$ and $\psi(z_0)=(\Ee,0)$. Now, by stability, one may consider the \'etale local family $(\mathbf{E},\mathbf{\Phi})$  (whose existence follows from Luna's slice theorem together with Lemma \ref{lm:no_nilp_par_end} and the arguments in Section \ref{sec:univ_bundles}). If $\Ee$ were strongly very stable, the generic point of this family would have to be very stable (by openness of very stability \cite[Proposition 3.5]{L}), contradicting the fact that $(\mathbf{E},\mathbf{\Phi})_z\in C_i\setminus C_0$ for $z\neq z_0$.
\end{proof}

\section{Shaky bundles are wobbly. The smooth case}\label{sec:W=S}

In this section we will assume that the parabolic weights are generic. In particular $\M(\Pp,\alpha)$ is smooth, so that all semistable parabolic (Higgs) bundles are stable; moreover, over the open subset $\M(\Pp,\alpha)^0\subset \M(\Pp,\alpha)$ of pairs with underlying stable parabolic bundle there exists a universal bundle (see Section \ref{sec:univ_bundles}).

By openness of semistability and irreducibility of $\M(\Pp,\alpha)$, there exists a rational map
\begin{equation}\label{eq:rat_full}
r:\xymatrix{\M(\Pp,\alpha)\ar@{-->}[r]&\N(\Pp,\alpha)}
\end{equation}{}
given by forgetting the Higgs field.  Note that  \eqref{eq:rat_full}  restricts to  rational maps
	\begin{equation}\label{eq:rat_nilp}
	r_{n}:\xymatrix{\M(\Pp,\alpha)_{nilp}\ar@{-->}[r]&\N(\Pp,\alpha)},
	\end{equation}
		\begin{equation}\label{eq:rat_st}
	r_{s}:\xymatrix{\M(\Pp,\alpha)_{st}\ar@{-->}[r]&\N(\Pp,\alpha)}.
	\end{equation}
	All the rational maps above are surjective, as by stability there is an embedding $\N(\Pp,\alpha)\plonge \M(\Pp,\alpha)_{st}$ given by $\Ee\mapsto(\Ee,0)$, which is in fact the composition  $\N(\Pp,\alpha)\plonge T^*\N(\Pp,\alpha)\plonge\M(\Pp,\alpha)_{st}$ of the zero section and the natural embedding.

Let $\Un\subset\M(\Pp,\alpha)$ be the subset given by semistable parabolic Higgs bundles with unstable underlying parabolic bundle. We denote by $\Un_{nilp}:=\Un\cap\M(\Pp,\alpha)_{nilp}$ and $\Un_{st}:=\Un\cap\M(\Pp,\alpha)_{st}$. 
\begin{remark}\label{rk:Un}
Non emptyness of $\Uns$ (and thus of $\Unn$ and $\Un$) follows, e.g., from the fact that the strongly parabolic nilpotent cone is reducible by \cite{L} (cf. Section  \ref{sec:rational_weights}). Indeed, by \cite[\S8]{SNonComp}, given $(\Ee,\varphi)\in {h^{st}}^{-1}(0)$, then $\lim_{t\to\infty} t\cdot (\Ee,\varphi)$ exists and is a fixed point for the $\CC^\times$-action, hence inside $h^{-1}(0)$. Since taking limits at $0$ and $\infty$ defines Zariski locally trivial affine fibrations with strata of half the dimension, limits at $\infty$ must have  unstable underlying bundle (this argument is found in \cite[\S 1.1]{Heinloth}).  
\end{remark}{}
\begin{remark}
In the non parabolic case, letting $\M$ denote the moduli space of Higgs bundles and $\N$ that of vector bundles, there is an equality $\Un:=\M\setminus T^*\N$. This is not the case here, as $T^*\N$ consists exclusively of strongly parabolic bundles  \cite[Remark 5.1]{YokoInf}. When $\Pp$ consists only of Borel subgroups, then $\Unn:=\Mn\setminus T^*\N(\Pp,\alpha)$.
\end{remark}{}
Given the smoothness assumption, Hironaka's results on the elimination of indeterminacies \cite{H} ensure that a finite number of blowups resolve the morphism \eqref{eq:rat_full} (see the discussion following \cite[Question E]{H}), which yields
\begin{equation}\label{eq:res_rat_full}
\xymatrix{
\widehat{\M(\Pp,\alpha)}\ar[dr]^{\hat{r}}\ar[d]_{\pi}&\\
\M(\Pp,\alpha)\ar@{-->}^r[r]&\N(\Pp,\alpha)
}
\end{equation}
Let $\Ex=\pi^{-1}(\Un)$ be the exceptional divisor of $\hM$. Then, since $\Ms$ and $\Mn$ are closed, it follows from \cite[Corollary II.7.15]{Ha} that $\hMn:=\widehat{\M(\Pp,\alpha)}\times_{\M(\Pp,\alpha)}\Mn$ and $\hMs:=\widehat{\M(\Pp,\alpha)}\times_{\M(\Pp,\alpha)}\Ms$ are closed subschemes of $\hM$ consisting of a finite number of blowups along $\Unn$ and $\Uns$ respectively. Hence we have resolutions 
\begin{equation}\label{eq:res_rat_nilp}
\xymatrix{
\widehat{\M(\Pp,\alpha)}_{nilp}\ar[dr]^{\hat{r}_n}\ar[d]_{\pi_n}&\\
\M(\Pp,\alpha)_{nilp}\ar@{-->}_{r_n}[r]&\N(\Pp,\alpha),
}
\end{equation}
\begin{equation}\label{eq:res_rat_st}
\xymatrix{
\widehat{\M(\Pp,\alpha)}_{st}\ar[dr]^{\hat{r}_s}\ar[d]_{\pi_s}&\\
\M(\Pp,\alpha)_{st}\ar@{-->}_{r_s}[r]&\N(\Pp,\alpha).
}
\end{equation}
In particular, the exceptional divisors  $\Ex_{nilp}:=\pi_n^{-1}(\Unn)$ and $\Ex_{st}:=\pi_s^{-1}(\Uns)$ satisfy $\Ex_{nilp}=\Ex\cap\hMn$ and $\Ex_{st}=\Ex\cap\hMs$.

\begin{definition}\label{defi:shaky}
	A stable bundle $\Ee$ is called shaky (resp. nilpotently shaky and strongly nilpotently shaky) if it is in the image $\Sh$ (resp.  $\Shn$ and $\Shs$)  under $\hat{r}$ (resp. $\hrn$ and $\hrs$) of the exceptional divisor $\Ex\subset{\hM}$ (resp. $\Ex_{nilp}$ and $\Ex_{st}$).
\end{definition}

\begin{theorem}\label{thm:W=S}
There is an equality $\Wo=\Shn=\Sh$, where $\Wo\subset\N(\Pp,\alpha)$ denotes the wobbly locus, namely, those stable bundles with a non-zero parabolic nilpotent Higgs field.

Similarly, the strong wobbly locus $\Wo_{st}$ satisfies $\Wo_{st}=\Sh_{st}$. 
\end{theorem}{}
\begin{proof}
We will prove that $\Wo\subset\Shn$ and that $\Sh\subset\Wo$. This proves the first statement, as $\Shn\subset\Sh$.

Let us first check that $\Wo\subset\Shn$. By Theorem \ref{thm:vs_iff_proper}, if $\Ee\in\Wo$, then $\Vn\subset\M(\Pp,\alpha)$ is not proper. So there exists a discrete valuation ring $R$ and a morphism 
$$
\iota_F:C:=\Spec(R)\longrightarrow \M(\Pp,\alpha)
$$
such that the generic point $\Spec(k)$ (where $k$ is the fraction field of $R$) maps to $\VE$ while the closed point $o$ goes to $(\Ff,\psi)\in \oVn\setminus{\Vn}$ (cf.  \cite[\href{https://stacks.math.columbia.edu/tag/0A24}{Tag 0A24}]{stacks-project}). By Lemma \ref{lm:oVn_minus_Vn}, $\Ff$ is unstable. 

Now, $\iota_F(C\setminus{\{o\}})=\iota_F(\Spec(k))$ can be seen as a subset of $\hMn$. More precisely, there is a commutative diagram
$$
\xymatrix{
\Spec(k)\ar[d]\ar[r]&\hMn\ar[d]_{\pi_n}\\
\Spec(R)\ar[r]_{\iota}\ar@{-->}[ur]^{\hat{\iota}}&\Mn.
}
$$
Properness of $\pi_n$ implies that $\hat{\iota}$ exists. Clearly, $\hat{\iota}(o)$ belongs to $\Ex_{nilp}$, and so $\hat{r}_n(\hat{\iota}(o))\in\Shn$. Consider the curve
$$
\hat{r}_n\circ\hat{\iota}:C\longrightarrow \N(\Pp,\alpha).
$$
Since $\iota(C\setminus\{o\})\subset \VE$, then $\hat{r}_n\circ\hat{\iota}(C\setminus\{o\})=\{\Ee\}$. Separability of $\N(\Pp,\alpha)$ implies that $\hat{r}_n\circ\hat{\iota}(o)=\Ee$, and hence $\Ee\in\Shn$.

For the converse: let $s\in\Sh$. Then, for some $S\in\Ex\subset\hM$, $s=\hat{r}(S)$. By \cite[Lemma 2.4]{PP}, one may find a smooth curve $Y$ and a morphism
$$
\xymatrix{j:Y\ar[r]&\hM\\
y\ar@{|->}[r]&(\Ee_y,\varphi_y)
}
$$
such that $j(Y\setminus\{y_0\})\subset\M(\Pp,\alpha)\setminus\Un$, and $j(y_0)=S$. Therefore
\begin{equation}
    \hat{r}\circ j(Y\setminus\{y_0\})\subset\N(\Pp,\alpha)\setminus \Sh, \qquad \hat{r}\circ j(y_0)=s.
\end{equation}{}
Consider now the rational map
$$
\xymatrix{
\CC\times Y\ar@{-->}[r]^m& \M(\Pp,\alpha)\\
(t,y)\ar@{|->}[r]&t\cdot \pi\circ j(y)
}$$
defined away from $(0,y_0)$. Note that 
\begin{enumerate}
    \item\label{it:curves} $m(t,Y\setminus\{y_0\})\subset\M(\Pp,\alpha)\setminus \Un$, as $t\cdot(\pi\circ j(y))=(\Ee_y,t\varphi_y)$, for all $y\neq y_0$, where $(\Ee_y,\varphi_y):=\pi\circ j(y))$.
    \item\label{it:curve_in_N_from_m} This implies $r(m(t,y))=r(m(0,y))=\hat{r}\circ j (y)$ for all $t\in \CC$ and all $y\neq y_0$.
\end{enumerate}
Now, by properness of $\N(\Pp,\alpha)$, it follows that ${r}\circ m(y_0)$ exists. By \eqref{it:curves} above, separatedness of $\N(\Pp,\alpha)$, and commutativity of \eqref{eq:res_rat_full}, ${r}\circ m(y_0)=s$.

  Moreover, by \eqref{it:curve_in_N_from_m}, $m(t,y_0)\in\Un$ for all $t\neq 0$.   Let $(\Ff_0,\psi_0)=\lim_{t\to 0}m(t,y_0)$. Assume that $\Ff_0$ were semistable. Then, so would be the bundle underlying $m(t,y_0)$ for generic $t\in\AA^1$  for generic $t$ \cite[Proposition 1.11]{Yoko}. But
  $m(t,y_0)=(\Ff,t\cdot \psi)$ has underlying unstable bundle, which yields a contradiction.
   Thus, it must be $(\Ff_0,\psi_0)\in\Un$.

Now, a finite number of blowups allows to resolve the morphism $m$ \cite{H}, yielding 
  $$
\hat{m}:\widehat{\CC\times Y}
\longrightarrow \M(\Pp,\alpha).
$$
The exceptional divisor over $(0,y_0)$ is a union of projective lines containing a curve joining $(\Ff_0,\psi_0)$ and $(\Ee,0)$ and contained inside $\Mn$ (as it is contained in $h_{\Pp,\alpha}^{-1}(0)\subset\Mn$). An irreducible  component $I$ of the exceptional divisor of $\widehat{\CC\times Y}$ must therefore intersect $\mathrm{N}(\Pp,\alpha)$ at $(\Ee,0)$, but also $h^{-1}(0)\setminus{\N(\Pp,\alpha)}$. We claim that  $I\cap\mathrm{N}(\Pp,\alpha)\subset \Wo$. Indeed, otherwise, it would intersect the very stable locus (as there are no strictly semistable bundles).  These are strongly very stable, which are a dense open set in $\N(\Pp,\alpha)$ by \cite[Proposition 3.5]{L}. Moreover, since $I\cap\M(\Pp,\alpha)^0$ is non empty, then it is dense, so one may conclude that bundles in $I$ with underlying strongly very stable bundle are dense. Thus, by irreducibility of $I$,   it would be $I\subset\N(\Pp,\alpha)$, contradicting the fact that $I$ intersects the complement of $\N(\Pp,\alpha)$.  In particular, $s\in I\cap\mathrm{N}(\Pp,\alpha)\subset \Wo$.

  To see that $\Wo_{st}=\Shs$, we note that the arguments adapt verbatim if we work in the strong nilpotent leaf.
\end{proof}{}

\begin{remark}
When the moduli space $\M(\Pp,\alpha)$ is not smooth,  Hironaka does not apply to the full moduli space. It however does to the reduced schemes underlying all Hitchin fibers. So working the analogue to Theorem \ref{thm:W=S} out requires a finer study of the Hitchin fibration, and the dependence of shakiness on the characteristic. Moreover, in order to compare all loci involved, strictly semitable points need to be tracked and discarded. 
\end{remark}{}
\section{Non-parabolic Higgs pairs}\label{sec:classical}
 This section shows that the results in the parabolic setup imply analogue ones in the usual (non parabolic) setup.
 \begin{definition} Let $D\subset X$ be a reduced divisor, possibly zero. A $D$-twisted Higgs pair is $(E,\varphi)$ where $E$ is a vector bundle and $\varphi\in H^0(X,\End(E)\otimes K(D))$. 
 \end{definition}
The moduli space $\M_D(n,d)$ of $D$-twisted Higgs pairs of rank $n$ and degree $d$ is a quasi-projective variety \cite{Nitin}. It is Poisson, with symplectic leaves given by the orbits of the residue of the Higgs field along the divisor \cite[Corollary 8.10]{Markman}, \cite[Theorem 4.7.5 ]{Bottacin}. Let $\N(n,d)$ denote the moduli space of vector bundles of rank $n$ and degree $d$.

Now, the notion of very stability is empty in the meromorphic setup, as the following result shows.
\begin{lemma}\label{lm:empty_D_vs}
    Every vector bundle has a nilpotent $D$-twisted Higgs field if $D\neq 0$.
\end{lemma}
\begin{proof}
     Indeed, by \cite[Theorem 0.1]{RT}, every semistable vector bundle is an extension
   \[
   F_0\hookrightarrow E\twoheadrightarrow F_1
   \] 
   with $0\leq d_1n_0-n_1d_0\leq n_0n_1(g-1)$ where $n_i=\rk(F_i)$, $d_i=\deg(F_i)$. But then, by Riemann--Roch, it follows that
   \[
   h^0(F_1^*F_0K(D))\geq n_0n_1\deg(D).
   \]
\end{proof}
\begin{remark}
    Thanks to C. Pauly for explaining the rank two case to me.
\end{remark}
Now, by Lemma \ref{lm:empty_D_vs}, it is possible to stratify the moduli space in terms of the existence of nilpotent Higgs bundles, with deep connections to stratification in terms of the Segre invariant. Motivated by this, we introduce the following notion.
\begin{definition}
Let $P\leq \GL(n,\CC)$ be a parabolic subgroup, $D$ a divisor. A vector bundle $E$ is called $(P,D)$-strongly very stable if it has no $D$-twisted nilpotent Higgs field inducing a reduction of the structure group to $P$ with Higgs field taking values in the nilpotent endomorphisms with respect to the parabolic subgroup.  
\end{definition}
\begin{lemma}\label{lm:D-vs_vs_par_vs}
  Let  $D$ be a reduced divisor, and let $\Pp=\{P\,:\, x\in D\}$, for some parabolic group $P$, so that all flags $\sigma$ have the same underlying parabolic group.
 The quasi-parabolic bundle $\Ee=(E,\sigma)$ is strongly very stable  if and only if $E$ is $(P,D)$-strongly very stable. In particular, if $P=\GL(n,\CC)$,  $\Ee$ is strongly very stable if and only if $E$ is very stable.
\end{lemma}
\begin{proof} The first statement is a tautological remark. 

The last statement follows easily from the fact that $H^0(X,\PEnd(\Ee)\otimes K(D))_{nilp}=H^0(X,\End(E)\otimes K(D))_{nilp}\supset H^0(X,\End(E)\otimes K)=H^0(X,\SPEnd(\Ee)\otimes K(D))$ and preservation of nilpotency under inclusion given the existence of a commutative  diagram 
$$
\xymatrix{
	H^0(X,\End(E)\otimes K)\ar@{^(->}[r]\ar[d]_h&H^0(X,\PEnd(\Ee)\otimes K(D))_{nilp}\ar[d]^{h_{\Ee,nilp}}\\
	\H_{0}\ar@{^(->}[r]&\H_D}.
$$
\end{proof}
As corollaries to Lemma \ref{lm:D-vs_vs_par_vs} we have the following.
\begin{corollary}\label{cor:vs_st_D}
    With the notation of Lemma \ref{lm:D-vs_vs_par_vs}, a $(P,D)$-strongly very stable $E$ bundle is stable.
\end{corollary}
\begin{proof}
Let  $\Pp=\{P\,:\, x\in D\}$.
Taking $\alpha=0$, we note that the moduli space $\N(\Pp,0)$ surjects  onto the moduli space of vector bundles $\N(n,d)$ (where $d$ equals the parabolic degree).  

Thus, the statement follows from Lemmas \ref{lm:D-vs_vs_par_vs} and  \ref{lm:vs_is_stable}.
\end{proof}
\begin{corollary}\label{cor:usual_criterion}
	Let $E$ be a stable vector bundle. Let $D\subset X$ be a reduced divisor, possibly zero, and let $\Pp=\{P\,:\, x\in D\}$ if $D\neq 0$, and $\Pp=\{\GL(n,\CC) \}$ if $D=0$. Let $\V_{E,D,n}:=H^0(X,\End(E)\otimes K(D))_{\mathfrak{n}}\subset H^0(X,\End(E)\otimes K(D))$ be the subset with residue in $\mathfrak{n}$. Consider the restrictions of the Hitchin map $h_{E,D,n}:=h|_{V_{E,D,n}}$.
	Then:
	$$
	\begin{array}{lcl}
	E \textnormal{ \it is  $(P,D)$-strongly very stable } &\iff&  h_{E,D,n} \textnormal{ \it is finite }\\
	&\iff& h_{E,D,n} \textnormal{ \it is quasi-finite }\\
	&\iff&  h_{E,D,n} \textnormal{ \it is proper }\\
	&\iff&  V_{E,D,n}\hookrightarrow \M_D(n,d) \textnormal{ \it is proper .}
	\end{array}
	$$
\end{corollary}{}
\begin{remark}
    If $P=\GL(n,\CC)$ $\V_{E,D,n}=\V_{E,D,0}=H^0(X,\End(E)\otimes K)$, regardless of the emptiness of the divisor.
\end{remark}
\begin{proof}
$E$ is stable if and only if $\Ee$ is stable. Then,
 by Lemma \ref{lm:D-vs_vs_par_vs}, $E$ is  $(P,D)$-strongly very stable if and only if $\Ee\in\Nn(\Pp,0)$ is strongly very stable. By Theorem \ref{thm:st_vs_iff_proper}, this is equivalent to $\V_{\Ee,st}\subset \M(\Pp,0)_{st}$ being proper. Now, since $\alpha=0$, there exist a map $\pi:\M(\Pp,0)_{st}\longrightarrow\M_D(n,d)$ which is proper. Thus, its restriction to $\V_{\Ee,st}\cong\V_{E,D,n}$ is also proper. 
 Conversely, assume  $\V_{E,D,n}\hookrightarrow\M_D(n,d)$ were proper.
Let $R$ be a discrete valuation ring with fraction field $k$. We want to prove the existence od the dashed arrow in the commutative diagram \[
\xymatrix{
\Spec(k)\ar[d]\ar[r]&\V_{\Ee,st}\ar[d]_i\\
\Spec(R)\ar[r]\ar@{-->}[ur]^{\hat{\iota}}&\M(\Pp,0).
}
\]
Since there exists a commutative diagram 
 \[
\xymatrix{
\V_{\Ee,st}\ar[r]\ar[d]^{\cong}&\M(\Pp, 0)\ar[d]\\
\V_{E,D, n}\ar[r]&\M_D(n,d),
}
 \]
further composing $\V_{\Ee,st}\subset\M(\Pp,0)$ with $\pi:\M(\Pp,0)\longrightarrow\M_D(n,d)$, by properness of $\V_{E,D,n}\cong \V_{\Ee,st}\subset\M_D(n,d)$ we obtain the existence of 
\[
\xymatrix{
\Spec(k)\ar[d]\ar[r]&\V_{\Ee,st}\ar[d]_i\\
\Spec(R)\ar[dr]\ar[r]\ar@{-->}[ur]^{\hat{\iota}}&\M(\Pp,0)\ar[d]\\
&M_D(n,d).
}
\]
 By commutativity, this factorisation proves the valuative criterion for properness for $\V_{\Ee,st}\longrightarrow \M(\Pp,0)$. Then, by Theorem \ref{thm:st_vs_iff_proper}, $\Ee$ is strongly very stable, which in turn is equivalent to $E$ being $(P,D)$-strongly very stable. This proves the equivalence between $(P,D)$-strong very stability and properness of $\V_{E,D, n}\longrightarrow \M_D(n,d)$. The proof of the equivalence of the latter and quasi-finiteness, properness and finiteness of $h_{E,D,n}$ follows as in the proof of Theorem \ref{thm:st_vs_iff_proper}. 
\end{proof}
\begin{remark}
	When $D=0$, Corollary \ref{cor:usual_criterion} is \cite[Theorem 1.1]{PP}. 
\end{remark}
Denote by $\N(n,d)^{vs}$ the very stable locus, and let $\mathbf{W}:=\N(n,d)^s\setminus \N(n,d)^{vs}$ be the wobbly locus. 
From Lemma \ref{lm:wobbly}, making $\alpha=0$, one gets
\begin{corollary}
	Let $C_i\, i\in I$ be the irreducible components of the nilpotent cone of $\M(n,d)$, with $C_0=\N(n,d)$. Then,  $\mathbf{W}=\bigcup_{i\in I\setminus 0}C_i\cap C_0^s$.
\end{corollary}{}

Likewise, $\hrs$ in\eqref{eq:res_rat_st} provides a resolution (when $(n,d)=1$) of the rational map
\begin{equation}\label{eq:res_usual}
\xymatrix{r:\M(n,d)\ar@{-->}[r]&\N(n,d)}.
\end{equation}
Then, $\Shs$  from Definition \ref{defi:shaky} is identified with a certain locus in  $\N(n, d)$ that we will denote by $\mathbf{S}$, image of the exceptional divisor over the locus $\mathbf{Un}\subset\M(n,d)$ of Higgs bundles with underlying unstable bundle. 

Finally, we also obtain the equivalent of Theorem \ref{thm:W=S}.  Assume that $(n,d)=1$ (equivalently, the weight $0$ is generic). 
\begin{theorem}\label{thm:usual_W=S}
There is an equality $\mathbf{W}=\mathbf{S}$.
\end{theorem}{}
\begin{remark}
Strictly speaking, from Theorem \ref{thm:W=S}, one can deduce that for any resolution of the rational map \eqref{eq:rat_st} obtained by restricting a resolution of \eqref{eq:rat_full}, the corresponding strongly shaky locus equals the strongly wobbly locus. We note that since we are only concerned with the images of the exceptional divisors, the same result will hold if instead of a restriction one considers arbitrary resolutions of \eqref{eq:rat_st} by successive blowups along $\mathbf{Un}$.
\end{remark}{}

\end{document}